\documentclass[11pt]{amsart}
\usepackage{mathrsfs,latexsym,amsfonts,amssymb,graphicx}
\newtheorem{theorem}{Theorem}[section]
\newtheorem{lemma}[theorem]{Lemma}
\newtheorem{corollary}[theorem]{Corollary}
\newtheorem{question}[theorem]{Question}

\theoremstyle{definition}
\newtheorem{definition}[theorem]{Definition}
\newtheorem{proposition}[theorem]{Proposition}
\theoremstyle{remark}

\begin{document}

\title[Topological gyrogroups with $Fr\acute{e}chet$-Urysohn property and $\omega^{\omega}$-base]
{Topological gyrogroups with $Fr\acute{e}chet$-Urysohn property and $\omega^{\omega}$-base}

\author{Meng Bao}
\address{(Meng Bao): College of Mathematics, Sichuan University, Chengdu 610064, P. R. China}
\email{mengbao95213@163.com}

\author{Xiaoyuan Zhang}
\address{(Xiaoyuan Zhang): College of Mathematics, Sichuan University, Chengdu 610064, P. R. China}
\email{405518791@qq.com}

\author{Xiaoquan Xu*}
\address{(Xiaoquan Xu): School of mathematics and statistics,
Minnan Normal University, Zhangzhou 363000, P. R. China}
\email{xiqxu2002@163.com}

\thanks{The authors are supported by the National Natural Science Foundation of China (11661057, 12071199) and the Natural Science Foundation of Jiangxi Province, China (20192ACBL20045)\\
*corresponding author}

%insert keywords
%insert subject class

\begin{abstract}
The concept of topological gyrogroups is a generalization of a topological group. In this work, ones prove that a topological gyrogroup $G$ is metrizable iff $G$ has an $\omega^{\omega}$-base and $G$ is Fr\'echet-Urysohn. Moreover, in topological gyrogroups, every (countably, sequentially) compact subset being strictly (strongly) Fr\'echet-Urysohn and having an $\omega^{\omega}$-base are all weakly three-space properties with $H$ a closed $L$-subgyrogroup.
\end{abstract}

\maketitle
\section{Introduction}
The concept of a gyrogroup was originally posed by Ungar in \cite{UA1988,UA}. It is obvious that a group is a gyrogroup that every groupoid automorphism is an identity mapping. Then, in 2017, Atiponrat \cite{AW} equipped the gyrogroup with a topology and gave the definition of a topological gyrogroup. At the same time, she gave some examples of topological gyrogroups, such as M\"{o}bius gyrogroups  equipped with standard topology. Moreover, she posed an open problem whether the first-countability and metrizability are equivalent in topological gyrogroups. Afterwards, Cai, Lin and He in \cite{CZ} gave a positive answer about this problem, since all topological gyrogroups are rectifiable spaces. In fact, this kind of spaces has been studied for many years, see \cite{AW1,AW2020,BL3,FM2,LF,LF1,LF3,UA,WAS2020}. However, they all did not research the quotient spaces of topological gyrogroups. Until in \cite{BL,BL1,BL2,BLX}, Bao and Lin started to investigate the quotient spaces of strongly topological gyrogroups and achieved some good results. For example, if $H$ is an admissible $L$-subgyrogroup of a strongly topological gyrogroup $G$, then the left coset space $G/H$ is submetrizable. More important, they constructed a strongly topological gyrogroup with an infinite $L$-subgyrogroup. By the same construction, we can obtain a topological gyrogroup with an infinite $L$-subgyrogroup. Therefore, it is meaningful to research the quotient spaces of a topological gyrogroup when the left coset is an $L$-subgyrogroup. In particular, we will investigate what properties of topological groups still hold in topological gyrogroups.

This paper is aims to research topological gyrogroups with $\omega^{\omega}$-base, Fr\'echet-Urysohn properties and weakly three-space properties. We prove that a topological gyrogroup $G$ is metrizable iff $G$ has an $\omega^{\omega}$-base and $G$ is Fr\'echet-Urysohn. Moreover, in topological gyrogroups, every (countably, sequentially) compact subset being strictly (strongly) Fr\'echet-Urysohn and having an $\omega^{\omega}$-base are all weakly three-space properties with $H$ a closed $L$-subgyrogroup. More precisely, if $H$ is a closed $L$-subgyrogroup of a topological gyrogroup $G$ with the first-countablility of every (countably, sequentially) compact subset of $H$, and if every (countably, sequentially) compact subset of $G/H$ is strictly (strongly) Fr\'echet-Urysohn, respectively, then every (countably, sequentially) compact subset of $G$ is strictly (strongly) Fr\'echet-Urysohn. Finally, if a topological gyrogroup $G$ has a closed first-countable $L$-subgyrogroup $H$ with the quotient space $G/H$ having an $\omega^{\omega}$-base, then $G$ has an $\omega^{\omega}$-base.

\section{Preliminaries}

In this paper, we assume that all topological spaces are Hausdorff, $\mathbb{N}$ denote the set of all positive integers and $\omega$ denote the first infinite ordinal. The readers see \cite{AA,E,UA} for more notation and terminology. Next we recall some definitions and facts.

\begin{definition}\cite{UA}
Assume that $(G, \oplus)$ is a groupoid. We call $(G,\oplus)$ a {\it gyrogroup}, if the following conditions are satisfied:

\smallskip
(G1) For every $a\in G$, there is a unique identity element $0\in G$ with $0\oplus a=a=a\oplus0$;

\smallskip
(G2) for every $x\in G$, we can find a unique inverse element $\ominus x\in G$ with $\ominus x \oplus x=0=x\oplus (\ominus x)$;

\smallskip
(G3) for every $x, y\in G$, we can find $\mbox{gyr}[x, y]\in \mbox{Aut}(G, \oplus)$ with $x\oplus (y\oplus z)=(x\oplus y)\oplus \mbox{gyr}[x, y](z)$ for arbitrary $z\in G$, and

\smallskip
(G4) for every $x, y\in G$, $\mbox{gyr}[x\oplus y, y]=\mbox{gyr}[x, y]$.
\end{definition}

Notice that a group is a gyrogroup naturally only if the $\mbox{gyr}[x,y]$ is the identity mapping.

\begin{definition}\cite{ST}
Assume that $(G,\oplus)$ is a gyrogroup. We call a nonempty subset $H$ of $G$ a {\it subgyrogroup}, if the restriction of $gyr[a,b]$ to $H$ is an automorphism of $H$ for every $a,b\in H$ and $H$ forms a gyrogroup under the operation inherited from $G$. Denote it by $H\leq G$.

\smallskip
Moreover, we call $H$ an {\it $L$-subgyrogroup}, if $H$ is a subgyrogoup and satisfies $gyr[a, h](H)=H$ for each $a\in G$ and $h\in H$. Denote it by $H\leq_{L} G$.
\end{definition}

\begin{lemma}\cite{UA}\label{a}
Assume that $(G, \oplus)$ is a gyrogroup. For every $x, y, z\in G$,

\begin{enumerate}
\smallskip
\item $(\ominus x)\oplus (x\oplus y)=y$. \ \ \ (left cancellation law)

\smallskip
\item $(x\oplus (\ominus y))\oplus gyr[x, \ominus y](y)=x$. \ \ \ (right cancellation law)

\smallskip
\item $(x\oplus gyr[x, y](\ominus y))\oplus y=x$.

\smallskip
\item $gyr[x, y](z)=\ominus (x\oplus y)\oplus (x\oplus (y\oplus z))$.
\end{enumerate}
\end{lemma}

\begin{definition}\cite{AW}
Call $(G, \tau, \oplus)$ a {\it topological gyrogroup} if it satisfies the followings:

\smallskip
(1) $(G, \tau)$ is a topological space.

\smallskip
(2) $(G, \oplus)$ is a gyrogroup.

\smallskip
(3) The binary operation $\oplus: G\times G\rightarrow G$ is jointly continuous, where $G\times G$ is equipped with the product topology, and the inverse operation $\ominus (\cdot): G\rightarrow G$, i.e. $x\rightarrow \ominus x$, is also continuous.
\end{definition}

$\mathbf{Remark}$ Obviously, every topological group is a topological gyrogroup. However, every topological gyrogroup whose gyrations are not identically equal to the identity is not a topological group. Moreover, it was given an example in \cite{AW} to show that there exists a topological gyrogroup but not a topological group, such as the Einstein gyrogroup equipped with the standard topology.

\begin{definition}\cite{BT,LPT,GK}
Let $X$ be a topological space and $x\in X$. We say that $x$ has a {\it neighborhood $\omega^{\omega}$-base} or a {\it local $\mathfrak{G}$-base} if there is a base $\{U_{\alpha}(x):\alpha \in \mathbb{N}^{\mathbb{N}}\}$ of neighborhoods at $x$ with $U_{\beta}(x)\subset U_{\alpha}(x)$ for every $\alpha \leq \beta$ in $\mathbb{N}^{\mathbb{N}}$, where $\mathbb{N}^{\mathbb{N}}$ consisted by all functions from $\mathbb{N}$ to $\mathbb{N}$ equipped with the natural partial order, ie., $f\leq g$ iff $f(n)\leq g(n)$ for every $n\in \mathbb{N}$. We say that $X$ has an {\it $\omega^{\omega}$-base} or a {\it $\mathfrak{G}$-base} if $X$ has a neighborhood $\omega^{\omega}$-base or a local $\mathfrak{G}$-base at each point of $X$.
\end{definition}

\begin{definition}\cite{FS}
If $X$ is a topological space and $A\subset X$, we call $A$ {\it sequentially closed} if there is not sequence of points of $A$ converging to a point not in $A$. Call $X$ {\it sequential} if all sequentially closed subsets of $X$ are closed.
\end{definition}

\begin{definition}\cite{FS}
A topological space $X$ is called {\it Fr\'echet-Urysohn at a point $x\in X$} if there exists a sequence $\{x_{n}\}_{n\in \mathbb{N}}$ in $A$ with $\{x_{n}\}_{n\in \mathbb{N}}$ converging to $x$ for each $A\subset X$ satisfying $x\in \overline{A}\subset X$. We call $X$ {\it Fr\'echet-Urysohn} if it is Fr\'echet-Urysohn at every point.
\end{definition}

\begin{definition} \cite{GJ,SF}
A topological space $X$ is called {\it strictly (strongly) Fr\'echet-Urysohn at a point $x\in X$} if whenever $\{A_{n}\}_{n\in \mathbb{N}}$ is a sequence (decreasing sequence) of subsets in $X$ and $x\in \bigcap _{n\in \mathbb{N}}\overline{A_{n}}$, we can find $x_{n}\in A_{n}$ for all $n\in \mathbb{N}$ with the sequence $\{x_{n}\}_{n\in \mathbb{N}}$ converging to $x$. We call $X$ {\it strictly (strongly) Fr\'echet-Urysohn} if it is strictly (strongly) Fr\'echet-Urysohn at every point.
\end{definition}

\section{$\omega^{\omega}$-base and Fr\'echet-Urysohn property in topological gyrogroups.}

In this section, ones research topological gyrogroups with $\omega^{\omega}$-base and Fr\'echet-Urysohn property. It is shown that if a topological gyrogroup $G$ is first-countable, then it has an $\omega^{\omega}$-base. If a topological gyrogroup $G$ has an $\omega^{\omega}$-base and is Fr\'echet-Urysohn, then it is metrizable. Therefore, we deduce that a topological gyrogroup $G$ is metrizable iff having an $\omega^{\omega}$-base and being Fr\'echet-Urysohn are both satisfied by $G$.

\bigskip
Suppose that a topological gyrogroup $G$ has an $\omega^{\omega}$-base $\{U_{\alpha}:\alpha \in \mathbb{N}^{\mathbb{N}}\}$. Set $$I_{k}(\alpha)=\{\beta \in \mathbb{N}^{\mathbb{N}}:\beta _{i}=\alpha _{i} ~for~i=1,...,k\}, ~~and~~D_{k}(\alpha)=\bigcap_{\beta \in I_{k}(\alpha)}U_{\beta},$$ where $\alpha =(\alpha _{i})_{i\in \mathbb{N}}\in \mathbb{N}^{\mathbb{N}}$ and $k\in \mathbb{N}$. Then $\{D_{k}(\alpha)\}_{k\in \mathbb{N}}$ is an increasing subset sequence of $G$ and contains the identity element $0$.

\begin{definition}\cite{CZ1}
We call a topological space $X$ {\it strong $\alpha_{4}$-space} if an arbitrary subset $\{x_{p,q}:p,q\in \mathbb{N}\}$ of $X$ is such that $lim_{q\rightarrow \infty}x_{p,q}=x\in X$ for each $p\in \mathbb{N}$, we can find strictly increasing natural number sequences $\{l_{k}\}_{k\in \mathbb{N}}$ and $\{m_{k}\}_{k\in \mathbb{N}}$ such that $lim _{k\rightarrow \infty}x_{l_{k},m_{k}}=x$.
\end{definition}

\begin{lemma}\label{fyl1}
Every first-countable topological gyrogroup has an $\omega^{\omega}$-base.
\end{lemma}

\begin{proof}
Let $G$ be a first-countable topological gyrogroup, and $\{V_{n}\}_{n\in \mathbb{N}}$ a decreasing base at the identity element $0$.  Put $W_{\alpha}=V_{\alpha_{1}}$ for each $\alpha \in \mathbb{N}^{\mathbb{N}}$. We have that $\{W_{\alpha}\}_{\alpha \in \mathbb{N}^{\mathbb{N}}}$ is an $\omega^{\omega}$-base in $G$.
\end{proof}

\begin{lemma}\cite{LF1}
If a topological gyrogroup $G$ is Fr\'echet-Urysohn, then it is a strong $\alpha_{4}$-space.
\end{lemma}

\begin{lemma}\cite{GKL}\label{GKL1}
Suppose $\alpha =(\alpha _{i})_{i\in \mathbb{N}}\in \mathbb{N}^{\mathbb{N}}$ and $\beta _{k}=(\beta ^{k}_{i})_{i\in \mathbb{N}}\in I_{k}(\alpha)$ for all $k\in \mathbb{N}$. Then we can find $\gamma \in \mathbb{N}^{\mathbb{N}}$ with $\alpha \leq \gamma$ and $\beta _{k}\leq \gamma$ for arbitrary $k\in \mathbb{N}$.
\end{lemma}

\begin{theorem}\label{fdl1}
If a Hausdorff topological gyrogroup $G$ is Fr\'echet-Urysohn and has an $\omega^{\omega}$-base $\{U_{\alpha}:\alpha \in \mathbb{N}^{\mathbb{N}}\}$, then $G$ is metrizable.
\end{theorem}

\begin{proof}
First, we establish the following:

{\bf Claim.} $D_{k}(\alpha)$ is a neighborhood of $0$ for each $\alpha \in \mathbb{N}^{\mathbb{N}}$ for some $k\in \mathbb{N}$.

Suppose not, that is, we can find $\alpha \in \mathbb{N}^{\mathbb{N}}$ such that $D_{k}(\alpha)$ is not a neighborhood of $0$, for any $k\in \mathbb{N}$. It means that $0\in \overline{G\setminus D_{k}(\alpha)}$ for any $k\in \mathbb{N}$. By the hypothesis, $G$ is Fr\'echet-Urysohn, for arbitrary $k$. Therefore, we can find a sequence $\{x_{n,k}\}_{n\in \mathbb{N}}$ in $G\setminus D_{k}(\alpha)$ which converges to $0$. Since all Fr\'echet-Urysohn Hausdorff topological gyrogroups are strong $\alpha _{4}$-spaces, we can choose natural numbers sequences $(n_{i})_{i\in \mathbb{N}}$ and $(k_{i})_{i\in \mathbb{N}}$ with $lim_{i\rightarrow \infty}x_{n_{i},k_{i}}=0$, where $(n_{i})_{i\in \mathbb{N}}$ and $(k_{i})_{i\in \mathbb{N}}$ are strictly increasing.

For each $i\in \mathbb{N}$, we choose $\beta _{k_{i}}\in I_{k_{i}}(\alpha)$ with $x_{n_{i},k_{i}}\not \in U_{\beta _{k_{i}}}$. It follows from Lemma \ref{GKL1} that, for every $i\in \mathbb{N}$, $\beta _{k_{i}}\leq \gamma$ for some $\gamma \in \mathbb{N}^{\mathbb{N}}$. Therefore, for any $i\in \mathbb{N}$, $x_{n_{i},k_{i}}\not \in U_{\gamma}$. We conclude that the sequence $\{x_{n_{i},k_{i}}\}_{i\in \mathbb{N}}$ does not converge to $0$ and this is a contradiction.

Therefore, for each $\alpha \in \mathbb{N}^{\mathbb{N}}$, set $D_{k_{\alpha}}(\alpha)$ is a neighborhood of $0$, where $k_{\alpha}$ is a minimal natural number. It is clear that $D_{k_{\alpha}}(\alpha)\subset U_{\alpha}$. Moreover, for $i\in \mathbb{N}$, fix $\alpha ^{(i)}=(i,\alpha_{2},\alpha_{3},...)\in \mathbb{N}^{\mathbb{N}}$. Then for any $\beta =(\beta_{1},\beta_{2},...)\in I_{1}(\alpha ^{(i)})$, $D_{1}(\beta)=D_{1}(\alpha ^{i})$. Therefore, $\{D_{1}(\alpha):\alpha \in \mathbb{N}^{\mathbb{N}}\}=\{D_{1}(\alpha^{(i)}):i\in \mathbb{N}\}$ is countable. So, $\{D_{k}(\alpha):k\in \mathbb{N},\alpha \in \mathbb{N}^{\mathbb{N}}\}$ is countable. Furthermore, $\{D_{k_{\alpha}}(\alpha):\alpha \in \mathbb{N}^{\mathbb{N}}\}\subset \{D_{k}(\alpha):k\in \mathbb{N},\alpha \in \mathbb{N}^{\mathbb{N}}\}$. Therefore, the countability of the family $\{D_{k_{\alpha}}(\alpha):\alpha \in \mathbb{N}^{\mathbb{N}}\}$ is obtained. In conclusion, the family $\{int(D_{k_{\alpha}}(\alpha)):\alpha \in \mathbb{N}^{\mathbb{N}}\}$ is a countable open neighborhood base at $0$. It follows from above that $G$ is first-countable and hence metrizable by \cite{CZ}.
\end{proof}

By Lemma \ref{fyl1} and Theorem \ref{fdl1},we deduce the following result.

\begin{corollary}
A topological gyrogroup $G$ is metrizable iff $G$ has an $\omega^{\omega}$-base and $G$ is also Fr\'echet-Urysohn.
\end{corollary}

Let $X$ be a topological space and $\mathcal{N}$ a family of subsets of $X$. We call $\mathcal{N}$ a {\it $cs^{*}$-network at a point $x\in X$} \cite{GMZ} if we can find $N\in \mathcal{N}$ with $x\in N\subset O_{x}$ and $\{n\in \mathbb{N}:x_{n}\in N\}$ is infinite, where $(x_{n})_{n\in \mathbb{N}}$ is arbitrary sequence in $X$ converging to $x$ and $O_{x}$ is arbitrary neighborhood of $x$.

Then we give the concept of $cs^{*}$-character in topological gyrogroups.

\begin{definition}
If $G$ is a topological gyrogroup, we call the least cardinality of $cs^{*}$-network at the identity element $0$ of $G$ {\it $cs^{*}$-character}.
\end{definition}

\begin{theorem}
If a topological gyrogroup $G$ has an $\omega^{\omega}$-base, then it has countable $cs^{*}$-character.
\end{theorem}

\begin{proof}
Suppose that $\{U_{\alpha}:\alpha \in \mathbb{N}^{\mathbb{N}}\}$ is an $\omega^{\omega}$-base in $G$. Put $\mathcal{D}=\{D_{k}(\alpha):\alpha \in \mathbb{N}^{\mathbb{N}},k\in \mathbb{N}\}$. Then $\mathcal{D}$  is countable and is a $cs^{*}$-network at $0$. Choose a sequence $S=(g_{n})_{n\in \mathbb{N}}$ in $G$ which converges to $0$ and fix a neighborhood $U_{\alpha}$ of $0$. Therefore, we just need to show that $S\cap D_{k}(\alpha)$ is infinite for some $k\in \mathbb{N}$.

{\bf Claim.} $S\cap D_{k}(\alpha)$ is infinite for some $k\in \mathbb{N}$.

Suppose on the contrary, $S\cap D_{k}(\alpha)$ is finite for every $k\in \mathbb{N}$. Then, for arbitrary $k\in \mathbb{N}$, take $n_{k}\in \mathbb{N}$ and $\beta_{k}\in I_{k}(\alpha)$ such that $n_{1}<n_{2}<...$ and $g_{n_{k}}\not \in U_{\beta_{k}}$. Then $\{g_{n_{k}}\}_{k\in \mathbb{N}}$ converges to $0$. By Lemma \ref{GKL1}, there exists $\gamma \in \mathbb{N}^{\mathbb{N}}$ with $\alpha \leq \gamma$ and $\beta_{k}\leq \gamma$ for all $k\in \mathbb{N}$. It is clear that $g_{n_{k}}\not \in U_{\gamma}$ for every $k\in \mathbb{N}$. Hence, $g_{n_{k}}\not \rightarrow 0$, which is a contradiction.
\end{proof}

It was claimed in \cite{BZ} that if $G$ is a Fr\'echet-Urysohn topological group and $G$ also has countable $cs^{*}$-character, then $G$ is metrizable. Now, we pose the following problem.

\begin{question}
If a topological gyrogroup $G$ is Fr\'echet-Urysohn and $G$ also has countable $cs^{*}$-character, is $G$ metrizable?
\end{question}

\section{The weakly three-space property in topological gyrogroups with Fr\'echet-Urysohn property}
In this section, we prove that every (countably, sequentially) compact subset being strictly (strongly) Fr\'echet-Urysohn is a weakly three-space property with $H$ a closed $L$-subgyrogroup of a topological gyrogroup $G$. More precisely, for a topological gyrogroup $G$ and a closed $L$-subgyrogroup $H$, if every (countably, sequentially) compact subset of $H$ is first-countable, and every (countably, sequentially) compact subset of $G/H$ is strictly (strongly) Fr\'echet-Urysohn, then every (countably, sequentially) compact subset of $G$ is strictly (strongly) Fr\'echet-Urysohn.

The following concept of the coset space of a topological gyrogroup was introduced in \cite{BL,BL2}.

If $H$ is an $L$-subgyrogroup of a topological gyrogroup $G$, it follows from \cite[Theorem 20]{ST} that $G/H=\{a\oplus H:a\in G\}$ forms a partition of $G$. Let $\pi:G\rightarrow G/H$ be $a\mapsto a\oplus H$ for all $a\in G$, then, it is obtained that $\pi^{-1}\{\pi(a)\}=a\oplus H$. Moreover, if we denote the topology of $G$ by $\tau (G)$, we define $\tau(G/H)$ in $G/H$ as the following: $$\tau (G/H)=\{O\subset G/H: \pi^{-1}(O)\in \tau (G)\}.$$

Furthermore, we call a property $\mathscr P$ {\it three-space property} in topological groups \cite{BT3,GKL} if topological group $G$ and a closed normal subgroup $H$ of $G$ both have $\mathscr P$, then $G$ enjoys $\mathscr P$, too.

Then we give the definitions of {\it three-space property} and {\it weakly three-sapce property} in topological gyrogroups.

\begin{definition}
We call a topological property $\mathscr Q$ {\it three-space property} in topological gyrogroups if a topological gyrogroup $G$ and a closed $L$-subgyrogroup $H$ of $G$ both have $\mathscr Q$, then $G$ enjoys $\mathscr Q$, too.
\end{definition}

\begin{definition}
We call a topological property $\mathscr Q$ {\it weakly three-space property} in topological gyrogroups, if in a topological gyrogroup $G$, there is a closed $L$-subgyrogroup $H$ of $G$ having a property $\mathscr P$ which is stronger than $\mathscr Q$ and the quotient space $G/H$ has $\mathscr Q$, then $G$ has $\mathscr Q$, too.
\end{definition}

Obviously, every three-space property in topological gyrogroups is a weakly three-space property in topological gyrogroups.

There is an open problem posed by A.V. Arhangel' ski\v\i $~~$and M. Tkachenko in \cite{AA}.

\begin{question}\cite[Open problem 9.10.3]{GG}
Let all compact subsets of the groups $H$ and $G/H$ be Fr\'echet-Urysohn. Does the same hold for compact subsets of $G$?
\end{question}

It is natural to pose the following problems.

\begin{question}\label{m}
Let $G$ be a topological gyrogroup and $H$ a closed $L$-subgyrogroup of $G$. If every compact subset of $H$ and $G/H$ both are Fr\'echet-Urysohn, is any compact subset of $G$ Fr\'echet-Urysohn? In particular, is any compact subset being Fr\'echet-Urysohn a weakly three-space property in topological gyrogroups?
\end{question}

\begin{question}\label{4m1}
Let $G$ be a topological gyrogroup and $H$ a closed $L$-subgyrogroup of $G$. If  every compact subset of $H$ and $G/H$ both are strictly (strongly) Fr\'echet-Urysohn, is any compact subset of $G$ strictly (strongly) Fr\'echet-Urysohn? In particular, is any compact subset being strictly (strongly) Fr\'echet-Urysohn a weakly three-space property in topological gyrogroups?
\end{question}

Next, we show that all compact subsets being strictly (strongly) Fr\'echet-Urysohn is a weakly three-space property in topological gyrogroups. More precisely, for a topological gyrogroup $G$ and a closed $L$-subgyrogroup $H$, if every (countably, sequentially) compact subset of $H$ is first-countable, and every (countably, sequentially) compact subset of $G/H$ is strictly (strongly) Fr\'echet-Urysohn, then every (countably, sequentially) compact subset of $G$ is strictly (strongly) Fr\'echet-Urysohn, and hence give a partial answer about Question \ref{4m1}, see Theorem \ref{n}.

Let $f: X\rightarrow Y$ be a continuous onto mapping and the space $Y$ and the fibers of $f$ have $\mathscr P$, then $X$ enjoys $\mathscr P$, we call the property $\mathscr P$ {\it inverse fiber property} \cite{BT}. Moreover, if the domain $X$ is (countably, sequentially) compact, we call $\mathscr P$ an {\it inverse fiber property for (countably, sequentially) compact sets}. Moreover, if the space $X$ is regular, call $\mathscr P$ {\it regular inverse fiber property}.

\begin{lemma}\cite{XL}\label{o}
The first-countability is an inverse fiber property for (countably, sequentilly) compact sets.
\end{lemma}

\begin{proposition}\label{a}
If $\mathscr P$ is an inverse fiber property, then it is a three-space property in topological gyrogroups.
\end{proposition}

\begin{proof}
We assume that $H$ is a closed $L$-subgyrogroup of a topological gyrogroup $G$. We assume further that both gyrogroups $H$ and $G/H$ have an inverse fiber property $\mathscr P$. If $y\in G/H$, we can find $x\in G$ such that $\pi (x)=y$. Then $\pi ^{-1}(y)=x\oplus H$ is homeomorphic with $H$, so the fiber $\pi^{-1}(y)$ has $\mathscr P$ for all $y\in G/H$. It follows from the inverse fiber property of $\mathscr P$ that $G$ enjoys $\mathscr P$, too.
\end{proof}

We call a topological space $X$ having a $G_{\delta}$-diagonal \cite{GG} if the diagonal $\Delta =\{(x,x): x\in X\}$ of $X\times X$ is a $G_{\delta}$-set in $X\times X$.

\begin{lemma}\label{e}
Let $G$ be a topological gyrogroup. The following two conditions are equivalent:

$(a)$ all sequentially compact subspaces of $G$ have the first axioms of countability;

\smallskip
$(b)$ all sequentially compact subspaces of $G$ are metrizable.
\end{lemma}

\begin{proof}
We only need to prove (a) $\Rightarrow$ (b). Let $X$ be a non-empty sequentially compact subset of $G$. Define a mapping $\varphi : G\times G\rightarrow G$ as $\varphi (x,y)=(\ominus x)\oplus y$ for every $x,y\in G$. Since $\varphi$ is continuous and $X\times X$ is sequentially compact, we have that $F=\varphi (X\times X)$ is sequentially compact subset of $G$ and $0\in F$. By the first-countability of $F$, $\{0\}$ is a $G_{\delta}$-set in $F$. Therefore, $(\varphi |_{X\times X})^{-1}(0)=\Delta$ is the diagonal in $X\times X$, and $\Delta$ is a $G_{\delta}$-set in $X\times X$. Moreover, it is well-known that every sequentially compact space is a countably compact space. Then, it follows from \cite[Theorem 2.14]{GG} that if $X$ is a countably compact space and $X$ has $G_{\delta}$-diagonal, then $X$ is compact and metrizable. Therefore, $X$ is metrizable.
\end{proof}

\begin{corollary}\label{p}
The following two conditions are equivalent in topological gyrogroups:

$(a)$ all countably compact subspaces have the first axioms of countability;

$(b)$ all countably compact subspaces are metrizable.
\end{corollary}

\begin{proof}
We will show that (a) $\Rightarrow$ (b). Indeed, it is known that if $X$ is countably compact and first-countable, then it is sequentially compact. Therefore, we complete the proof by Lemma \ref{e}.
\end{proof}

\begin{corollary}\label{g}
all of the following are three-space properties in topological gyrogroups:

$(a)$ every sequentially compact subset is closed.

\smallskip
$(b)$ every sequentially compact subset is compact.

\smallskip
$(c)$ every (countably, sequentially) compact subset is first-countable.

\smallskip
$(d)$ every (countably, sequentially) compact subset is metrizable.
\end{corollary}

\begin{proof}
By Proposition \ref{a} and \cite[Lemma 2.2]{LS}, (a) and (b) hold.

\smallskip
By Lemma \ref{o}, we have that every (countably, sequentially) compact subset satisfying the first axiom of countability is an inverse fiber property and by Proposition \ref{a}, Lemma \ref{e}, Corollary \ref{p} and \cite[Theorem 3.10]{LF}, (c) and (d) hold.
\end{proof}

\begin{theorem}
every sequentially compact subset being sequential is a three-space property in topological gyrogroups.
\end{theorem}

\begin{proof}
Assume that $G$ is a topological gyrogroup and $H$ is a closed $L$-subgyrogroup of $G$. Assume further that every sequentially compact subsets of both $H$ and $G/H$ are sequential. It follows from \cite[Lemma 2.4]{LS} that every sequentially compact subset of $H$ and $G/H$ both are closed. By Corollary \ref{g}, every sequentially compact subset of $G$ is closed. For arbitrary sequentially compact subset $B$ of $G$, let $A$ be a sequentially closed subset of $B$. It is clear that $A$ is sequentially compact in $G$. Therefore, $A$ is closed. We obtain that $B$ is sequential by the definition. In conclusion, every sequentially compact subset of $G$ is sequential since $B$ is arbitrary.
\end{proof}

\begin{lemma}\label{h}
If all (countably, sequentially) compact subspaces of a topological gyrogroup $G$ are Fr\'echet-Urysohn, then all (countably, sequentially) compact subspaces of $G$ are strongly Fr\'echet-Urysohn.
\end{lemma}

\begin{proof}
Let $G$ be a topological gyrogroup and all (countably, sequentially) compact subspaces of $G$ Fr\'echet-Urysohn. It follows from \cite[Lemma 2.4]{LS} that all (countably, sequentially) compact subsets of $G$ are closed. Let $A$ be an arbitrary (countably, sequentially) compact subset of $G$. We have that $A$ is Fr\'echet-Urysohn and closed. Let $a$ be an accumulation point of $A$ and let $\{A_{n}\}_{n\in \mathbb{N}}$ be a decreasing sequence of subsets of $A$ such that $a\in \bigcap _{n\in \mathbb{N}}\overline{A_{n}}$. Since $A$ is Fr\'echet-Urysohn, there exists a sequence $\{a_{n}\}_{n\in \mathbb{N}}$ in $A\backslash \{a\}$ converging to $a$. Set $$B = (\ominus a)\oplus A,~~and~~B_{n}=(\ominus a)\oplus A_{n},~~b_{n}=(\ominus a)\oplus a_{n}~~~~for~~all~~n\in \mathbb{N}.$$ Obviously, $B$ is closed, $0\in (\ominus a)\oplus \overline{A_{n}}=\overline{B_{n}}\subset B$, $b_{n}\in B\backslash \{0\}$ for every $n\in \mathbb{N}$ and $\{b_{n}\}_{n\in \mathbb{N}}$ converges to $0$. Let $\{V_{n}\}_{n\in \mathbb{N}}$ be a symmetric open neighborhood sequence of $0$ in $G$ with $b_{n}\not \in V_{n}\oplus V_{n}$ for each $n\in \mathbb{N}$. Put $C_{n}=(B_{n}\cap V_{n})\oplus b_{n}$  for every $n\in \mathbb{N}$. Since $0\in \overline{B_{n}\cap V_{n}}$, we have $b_{n}\in \overline{C_{n}}$. Moreover, it follows from $V_{n}\cap C_{n}\subset V_{n}\cap (V_{n}\oplus b_{n})=\emptyset$ that $0\not \in \overline{C_{n}}$.

Now set  $$D=\bigcup \{C_{n}: n\in \mathbb{N}\},~~and~~S=\{0\}\cup \{b_{n}: n\in \mathbb{N}\}.$$ Then $D\subset \bigcup _{n\in \mathbb{N}}(B_{n}\oplus b_{n})\subset B\oplus S$.

{\bf Claim.} The subspace $B\oplus S$ of $G$ is Fr\'echet-Urysohn and closed.

It is obvious that $S$ is compact and sequentially compact.

{\bf Case 1.} Suppose that $A$ is compact. By the compactness of $A$, $B$ is also compact. Therefore, the Cartesian product $B\times S$ is compact. Moreover, since the binary operation in $G\times G$ is jointly continuous, it is obtained that $B\oplus S$ is compact as the continuous image of $B\times S$. Thus, $B\oplus S$ is Fr\'echet-Urysohn and closed.

{\bf Case 2.}We assume that $A$ is countably compact or sequentially compact. If $A$ is countably compact, by the Fr\'echet-Urysohn property of $A$, we obtain that $A$ is sequentially compact. Since $L_{\ominus a}$ is homeomorphic, it is obtained that $B$ is also sequentially compact. Then the Cartesian product $B\times S$ is sequentially compact. Furthermore, since the binary operation in $G\times G$ is jointly continuous, it is achieved that $B\oplus S$ is sequentially compact as the continuous image of $B\times S$. Thus, $B\oplus S$ is Fr\'echet-Urysohn and closed.

Since $b_{n}\in \overline{C_{n}}$ for each $n\in \mathbb{N}$ and $b_{n}$ converges to $0$, we have that $0\in \overline{D} \subset B\oplus S$, and we can find a sequence $\{d_{k}\}_{k\in \mathbb{N}}$ converging to $0$ in $D$. For all $n\in \mathbb{N}$, since $0\not \in \overline{C_{n}}$, $C_{n}$ contains only finitely many terms of the sequence $\{d_{k}\}_{k\in \mathbb{N}}$. There exists a subsequence $\{C_{n_{k}}\}_{k\in \mathbb{N}}$ of the sequence $\{C_{n}\}_{n\in \mathbb{N}}$ with $d_{k}\in C_{n_{k}}$ for each $k\in \mathbb{N}$. Since $C_{n_{k}}\subset B_{n_{k}}\oplus b_{n_{k}}=((\ominus a)\oplus A_{n_{k}})\oplus b_{n_{k}}$, there exists $x_{n_{k}}\in A_{n_{k}}$ such that $d_{k}=((\ominus a)\oplus x_{n_{k}})\oplus b_{n_{k}}$, for all $k\in \mathbb{N}$. Then $(\ominus a)\oplus x_{n_{k}}=d_{k}\oplus gyr[(\ominus a)\oplus x_{n_{k}},b_{n_{k}}](\ominus b_{n_{k}})$, so $x_{n_{k}}=a\oplus (d_{k}\oplus gyr[(\ominus a)\oplus x_{n_{k}},b_{n_{k}}](\ominus b_{n_{k}}))$. Since $gyr[x,y](z)=\ominus (x\oplus y)\oplus (y\oplus (y\oplus z))$ for all $x,y,z\in G$, it follows that $gyr[x,y](0)=\ominus (x\oplus y)\oplus (x\oplus y)=0$. Therefore, $x_{n_{k}}=a\oplus (d_{k}\oplus gyr[(\ominus a)\oplus x_{n_{k}},b_{n_{k}}](\ominus b_{n_{k}}))\rightarrow a$ whenever $k\rightarrow \infty$. When $n_{k-1}< n\leq n_{k}$, fix $y_{n}=x_{n_{k}}$. Then $y_{n}\in A_{n}$ for each $n\in \mathbb{N}$ and $y_{n}\rightarrow a$. In conclusion, $A$ is strongly Fr\'echet-Urysohn.
\end{proof}

\begin{theorem}\label{n}
Let $H$ be a closed $L$-subgyrogroup of a topological gyrogroup $G$ with every (countably, sequentially) compact subset of $H$ being first-countable. If $G/H$ has one of the following conditions, then $G$ has the same property:

$(a)$ every (countably, sequentially) compact subset is strongly Fr\'echet-Urysohn.

$(b)$ every (countably, sequentially) compact subset is strictly Fr\'echet-Urysohn.

\end{theorem}

\begin{proof}
It follows from Theorem 3 in \cite{AW} that every $T_{0}$ topological gyrogroup is regular. Suppose that $C$ is a (countably, sequentially) compact subset of $G$. It follows from Proposition \ref{a}, \cite[Lemma 2.2]{LS} and \cite[Lemma 2.4]{LS} that $C$ is closed in $G$. Set $\varphi =\pi |_{C}:C\rightarrow \pi (C)$, we obtain that $\pi (C)$ is (countably, sequentially) compact. Moreover, $\varphi$ is a closed mapping by \cite[Lemma 2.4]{LS}, and $\varphi^{-1}(\varphi(c))=\pi ^{-1}(\pi (c))\cap C = (c\oplus H)\cap C$ is first-countable for every $c\in C$. We complete the proof by Lemma \ref{h}, \cite[Proposition 4.7.18]{AA} and \cite[Lemma 2.11]{LS}.
\end{proof}

\section{The weakly three-space property in topological gyrogroups  with $\omega^{\omega}$-base }

Finally, we study whether having an $\omega^{\omega}$-base is a (weakly) three-space property in topological gyrogroups. In particular, we discuss the following question.

\begin{question}\label{4wt1}
If a closed $L$-subgyrogroup $H$ of a topological gyrogroup $G$ and the quotient space $G/H$ both have an $\omega^{\omega}$-base, does $G$ have an $\omega^{\omega}$-base?
\end{question}

We will show that having an $\omega^{\omega}$-base is a weakly three-space property in topological gyrogroups. More precisely, if a closed $L$-subgyrogroup $H$ of a topological gyrogroup $G$ is first-countable and $G/H$ has an $\omega^{\omega}$-base, then $G$ has an $\omega^{\omega}$-base.

\begin{theorem}
If a closed $L$-subgyrogroup $H$ of a topological gyrogroup $G$  is first-countable and the quotient space $G/H$ has an $\omega^{\omega}$-base, then $G$ has an $\omega^{\omega}$-base.
\end{theorem}

\begin{proof}
Suppose that $\{W_{n}\}_{n\in \mathbb{N}}$ is a family of open symmetric neighborhoods of $0$ and is decreasing with $W_{n+1}\oplus W_{n+1}\subset W_{n}$ for each $n\in \mathbb{N}$. Since $H$ is first-countable, suppose further that $\{W_{n}\cap H\}_{n\in \mathbb{N}}$ is an open neighborhood base of $0$ in $H$. Assume that $\mathscr V=\{V_{\alpha}:\alpha \in \mathbb{N}^{\mathbb{N}}\}$ is an $\omega^{\omega}$-base in $G/H$. Without loss of generality, let $V_{\alpha}$ be symmetric for every $V_{\alpha} \in \mathscr V$. Set $$\{U_{\alpha}=\pi ^{-1}(V_{\alpha}):\alpha \in \mathbb{N}^{\mathbb{N}}\}.$$ Put $R_{\beta}=W_{n}\cap U_{\alpha}$, for each $n\in \mathbb{N}$ and $\alpha =(\alpha _{i})\in \mathbb{N}^{\mathbb{N}}$, where $\beta =(n,\alpha _{1},\alpha _{2},...)$. Then $\mathscr R=\{R_{\beta}:\beta \in \mathbb{N}\times \mathbb{N}^{\mathbb{N}}\}$ is a family of open symmetric neighborhoods of $0$ with $R_{\beta}\subset R_{\gamma}$ if $\beta \geq \gamma$.

{\bf Claim.} $\mathscr R$ is an $\omega^{\omega}$-base for $G$.

Since $G$ is homogenous, it suffices to show that $\mathscr R$ is an $\omega^{\omega}$-base at $0$ for $G$. Let $U$ be an arbitrary open neighborhood of $0$. We can find an open symmetric neighborhood $V$ of $0$ with $V\oplus V\subset U$. By the construction of $W_{n}$, there exists $n\in \mathbb{N}$ with $W_{n}\cap H\subset V$. Then, as $\mathscr V$ is a symmetric $\omega^{\omega}$-base in $G/H$, we can choose $V_{\alpha}\in \mathscr V$ such that $V_{\alpha}=\pi (U_{\alpha})\subset \pi (V\cap W_{n+1})$. Put $$R_{\beta}=W_{n+1}\cap \pi^{-1}(V_{\alpha})=W_{n+1}\cap U_{\alpha}.$$

For an arbitrary $g\in R_{\beta}$, since $R_{\beta}=W_{n+1}\cap U_{\alpha}$, $g\in W_{n+1}$ and $g\in U_{\alpha}\subset (V\cap W_{n+1})\oplus H$. There exist $a\in V\cap W_{n+1}$ and $b\in H$ satisfying $g=a\oplus b$. Therefore, $b=(\ominus a)\oplus g\in (W_{n+1})\oplus W_{n+1}\subset W_{n}$. So $g\in (V\cap W_{n+1})\oplus (H\cap W_{n})\subset V\oplus V\subset U$. In conclusion, $R_{\beta}\subset U$ and we complete the proof.
\end{proof}

\end{document}